\documentclass[final,twoside,a4paper]{amsart}
\usepackage[english]{babel}
\usepackage{amsmath, amsfonts,amssymb,amsopn,amscd,amsthm,mathrsfs}
\usepackage{bm}
\usepackage{graphicx}
 \usepackage{enumerate}
 \usepackage{array}
 \usepackage{hyperref}
 \hypersetup{colorlinks=true,linkcolor=blue,citecolor=blue,linktoc=page}
\usepackage{showkeys}
\numberwithin{equation}{section}
\usepackage{fancyhdr}
\newcommand{\eq}{\begin{equation}}
\newcommand{\qe}{\end{equation}}

\newcommand{\E}{\mathbb{E}}
\newcommand{\N}{\mathbb{N}}
\newcommand{\R}{\mathbb{R}}

\theoremstyle{plain}

\newtheorem{thm}{Theorem}[section]
\newtheorem{lem}{Lemma}[section]
\newtheorem{prop}{Proposition}[section]

\newtheorem{defn}{Definition}[section]
\theoremstyle{definition}

\theoremstyle{remark}
\newtheorem*{rem}{Remark}

\usepackage{color}

\begin{document}
\sloppy
\pagestyle{headings} 
\title{Remarks on Superconcentration and Gamma calculus. Applications to spin glasses}
\date{Note of \today}
\author{Kevin Tanguy \\ University of Angers, France}

\begin{abstract}
This note is concerned with the so-called superconcentration phenomenon. It shows that the Bakry-\'Emery's Gamma calculus can provide relevant bound on the variance of function satisfying a inverse, integrated, curvature criterion. As an illustration, we present some variance bounds for the Free Energy in different models from Spin Glasses Theory.
\end{abstract}
\maketitle 
\section{Introduction}

Superconcentration phenomenon has been introduced by Chatterjee in \cite{Chatt1} and has given birth to a lot a work (cf. \cite{KT1} for a survey). Each of these work, used various ad-hoc methods to improve upon sub-optimal bounds given by classical concentration of measure (cf. \cite{BLM,Led}). In this note, we want to show that the celebrated Gamma calculus from Bakry and \'Emery's Theory is relevant to such improvements. To this task, we introduce an inverse, integrated, Gamma two criterion which provides a useful bound on the variance of a particular function. As far as we are concerned, this criterion seems to be new. We give below a sample of our modest achievement.\\

 Denote by $\gamma_n$ the standard Gaussian measure on $\R^n$ and by $(P_t)_{t\geq 0}$ the standard Ornstein-Uhlenbeck semigroup.  $\Gamma$ will stand for the so-called "carr\'e du champ" operator, associated to the infinitesimal generator $L=\Delta-x\cdot \nabla$ of $(P_t)_{t\geq 0}$, and $\Gamma_2$ its iterated operator. We refer to section two for more details about this topic.

\begin{thm}
Let  $f\,:\,\R^n\to \R$ be a regular function and assume that there exists $\psi\,:\,\R_+\to\R$ such that

\begin{enumerate}
\item  for any $t\geq 0$, 

\begin{equation}\label{G1}
\int_{\R^n}\Gamma_2(P_t f)d\gamma_n\leq  \int_{\R^n}\Gamma(P_t f)d\gamma_n+\psi(t),
\end{equation}
\item
\[
\int_0^\infty e^{-2t}\int_t^\infty e^{2 s}\psi(s)ds dt<\infty.
\]
\end{enumerate}
Then the following holds

\[
{\rm Var}_{\gamma_n}(f)\leq \bigg|\int_{\R^n}\nabla fd\gamma_n\bigg|^2+4\int_0^\infty e^{-2 t}\int_t^\infty e^{2s}\psi(s)dsdt.
\]

\noindent with $|\cdot|$ the standard Euclidean norm.
\end{thm}
\begin{rem}
Equation \eqref{G1} can be seen as an inverse, integrated, curvature inequality for the function $f$.
\end{rem}

 As an application of Theorem \ref{prop.courbure.dimension.inverse}, we show that some results due to Chatterjee can be expressed in term of such criterion. From our point of view, this expression seems to ease the original scheme of proof and could possibly lead to various extensions. It also permits to easily recover some known variance bounds in Spin Glass Theory (cf. \cite{BovKurLow,Talspin1,Talspin2}).\\
 
 Denote by $F_{n,\beta}$ the Free Energy associated to the Sherrington and Kirkpatrick's Spin Glass Model (SK model in short) (cf. section four for more details about this). We show that
 
 \begin{prop}
 The following holds for the SK model. Let $0<\beta<\frac{1}{2}$ be, then
 
 \[
 {\rm Var}(F_{n,\beta})\leq \frac{C_\beta}{2\beta^2},\quad n\geq 1
 \]
 
 \noindent with $C_\beta>0$ is a constant depending only on $\beta$.
 \end{prop}

The methodology can also be used for the Random Energy Model (REM in short) (cf. section four for more details) and provides the following bounds.

\begin{prop}
The following holds in the REM.
\begin{enumerate}
\item High temperature regime : for $0<\beta<\sqrt{\frac{\ln 2}{2}}$, we have 

\[
{\rm Var}_{\gamma_n}(F_{n,\beta})\leq \bigg(\frac{1-\beta^2}{1-2\beta^2}\bigg)\frac{1}{n},\quad n\geq 1
\]

\item Low temperature regime : for $\beta\sim\sqrt{\log n}$, we have 
\[
{\rm Var}_{\gamma_n}(F_{n,\beta})\leq \frac{C}{\log n},\quad n\geq e
\]
\noindent with $C>0$ a universal constant.
\end{enumerate}
\end{prop}

This note is organized as follows. In section two, we recall some facts about superconcentration and Gamma calculus. In section three, we will prove our main results. Finally, in section four, we will give some applications in Spin Glass Theory.

\section{Framework and tools}

In this section, we briefly recall some notions about superconcentration, Gamma calculus and interpolation methods by semigroups. General references about these topics could be, respectively, \cite{Chatt1, BGL}.

\subsection{Superconcentration}

It is well known (cf. \cite{Led, BLM}), that concentration of measure of phenomenon is useful in various mathematical contexts. Such phenomenon can be obtained through functional inequalities. For instance, the standard Gaussian measure, on $\R^n$, $\gamma_n$ satisfies a Poincar\'e's inequality :

\begin{prop}
For any function $f\,:\,\R^n\to \R$ smooth enough, the following holds

\begin{equation}\label{eq.poincare}
{\rm Var}_{\gamma_n}(f)\leq \int_{\R^n}|\nabla f|^2d\gamma_n
\end{equation}

\noindent where $|\cdot|$ stands for the Euclidean norm.
\end{prop}

Although this inequality holds for a large class of function, it could lead to sub-optimal bounds. A classical example is the function $f(x)=\max_{i=1,\ldots,n}x_i$.  For such function, Poincar\'e's inequality implies that 

\[
{\rm Var}_{\gamma_n}(f)\leq 1
\]

\noindent but it is known that ${\rm Var}_{\gamma_n}(f)\sim \frac{C}{\log n}$ for some constant $C>0$. In Chatterjee's terminology, in this Gaussian framework, a function $f$ is said to be superconcentrated when Poincar\'e's inequality \eqref{eq.poincare} is sub-optimal. \\

As we have said in the introduction, this phenomenon has been studied in various manner : semigroup interpolation \cite{KT}, Renyi's representation of order statistics \cite{BT}, Optimal Transport \cite{KT1}, Ehrard's inequality \cite{Val},\ldots (cf. the Thesis \cite{KT2} for a recent survey about superconcentration). In this note, we want to show that some differential inequalities between the operator $\Gamma$ and $\Gamma_2$ from Bakry and \'Emery's Theory could provide superconcentration.

\subsection{Semigroups interpolation and Gamma calculus}

For more details about semigroups interpolation and $\Gamma$ calculus, we refer to \cite{BGL,Led2}. Although our work can easily be extended to a more general framework, we will focus on a Gaussian setting.\\

The Ornstein-Uhlenbeck process $(X_t)_{t\geq0}$ is defined as follow  :

$$
X_t=e^{-t}X+\sqrt{1-e^{-2t}}Y,\quad t\geq 0,
$$ 
\noindent with $X$ and $Y$ i.i.d. standard Gaussian vectors in $\R^n$. The semigroup $(P_t)_{t\geq0}$, associated to this process, acts on a class of smooth function and admits an explicit representation formula : 

$$
P_tf(x)=\int_{\R^n}f\big(xe^{-t}+\sqrt{1-e^{-2t}}y\big)d\gamma_n(y),\quad x\in\R^n,\,t\geq 0
$$

\noindent Its infinitesimal generator is given by 
 
$$
L=\Delta-x\cdot\nabla
$$

\noindent Furthermore, $\gamma_n$ is the invariant and reversible measure of $(P_t)_{t\geq 0}$. That is to say, for any function $f$ and $g$ smooth enough,

\[
\int_{\R^n} P_t fd\gamma_n=\int_{\R^n} fd\gamma_n \quad \text{et}\quad \int_{\R^n} fP_tgd\gamma_n=\int_{\R^n} gP_tfd\gamma_n.
\]

Now, let us recall some properties satisfied by $(P_t)_{t\geq 0}$ which will be useful in the sequel.

\begin{prop}
The Ornstein-Uhlenbeck semigroup $(P_t)_{t\geq 0}$ satisfies the following properties
\begin{enumerate}

\item[$\bullet$] $P_t(f)$ is a solution of the heat equation associated to $L$

\begin{equation}\label{eq.chaleur}
{\rm i.e.}\quad\partial_t (P_tf)=P_t (Lf)=L(P_tf).
\end{equation}

\item[$\bullet$] $(P_t)_{t\geq0}$ is ergodic, that is to say, for $f$ smooth enough

\begin{equation}\label{eq.ou.ergodicite}
\lim_{t\to+\infty}P_t(f)= \int_{\R^n}fd\gamma_n=\E_{\gamma_n}[f]
\end{equation}

\item[$\bullet$] $(P_t)_{t\geq 0}$ commutes with the gradient $\nabla$. More precisely, for any function  $f$ smooth enough, 

\begin{equation}\label{eq.commutation.gaussienne}
\nabla P_t(f)=e^{-t}P_t(\nabla f), \quad t\geq 0.
\end{equation}

\item[$\bullet$] $(P_t)_{t\geq 0}$ is a contraction in $L^p(\gamma_n)$, for any function $f\in L^p(\gamma_n)$ and every $t\geq 0$,

\begin{equation}\label{eq.ou.contraction}
\|P_t(f)\|_p\leq \|f\|_p.
\end{equation}


\end{enumerate}
\end{prop}

As it exposed in \cite{BGL}, it is possible to give a dynamical representation of the variance of a function $f$ along the semigroup $(P_t)_{t\geq 0}$ :

\begin{equation}\label{eq.representation.variance.gaussienne}
{\rm Var}_{\gamma_n}(f)=2\int_0^\infty \int_{\R^n}|\nabla P_s( f)|^2d\gamma_n ds=2\int_0^\infty e^{-2s}\int_{\R^n}|P_s(\nabla f)|^2d\gamma_n ds
\end{equation}

\subsection{Gamma calculus and Poincar\'e's inequality}

Let us introduce the fondamental operator $\Gamma_2$ and $\Gamma$ from Bakry and Emery's Theory. Given an infinitesimal generator $L$ set, for $f$ and $g$, two smooth functions, 

\[
\Gamma(f,g)=\frac{1}{2}\big[L(fg)-fLg-Lfg\big]\quad \text{and}\quad \Gamma_2(f,g)=\frac{1}{2}\big[L\Gamma(f,g)-\Gamma(f,Lg)-\Gamma(Lf,g)\big]
\] 

In the case of the Ornstein-Uhlenbeck's infinitesimal generator $L=\Delta-x\cdot \nabla$, it is easily seen that

\begin{equation}\label{eq.gamma.ou}
\Gamma_1(f)=|\nabla f|^2\quad \Gamma_2(f)=\|{\rm Hess} f\|_2^2+\|\nabla f\|^2
\end{equation}

\noindent where $\|{\rm Hess} f\|_2=\big(\sum_{i,j=1}^n\big(\frac{\partial^2f}{\partial x_i\partial x_j}\big)^2\big)^{1/2}$ is the Hilbert-Schmidt norm of the tensor of the second derivatives of $f$.\\

Now, let us briefly recall how a relationship between $\Gamma$ and $\Gamma_2$ can be used to give a elementary proof of Poincar\'e's inequality \eqref{eq.poincare}.\\

First, notice that the representation formula of the variance \eqref{eq.representation.variance.gaussienne} can be expressed in term of $\Gamma$ :

\begin{equation}\label{eq.variance.represantation2}
{\rm Var}_{\gamma_n}(f)=2\int_0^\infty \int_{\R^n}\Gamma(P_tf)d\gamma_n ds.
\end{equation}

Then, observe that \eqref{eq.gamma.ou} implies the celebrated curvature-dimension criterion $CD(1,+\infty)$ (cf. \cite{BGL})

\begin{equation}\label{eq.curvature.dimension}
\Gamma_2\geq \Gamma. 
\end{equation}

Set $I(t)=\int_{\R^n}\Gamma(P_tf)d\gamma_n$. It is classical that 

\[
I'(t)=-2\int_{\R^n}\Gamma_2(P_tf)d\gamma_n,\quad t\geq 0
\]

Thus, the inequality \eqref{eq.curvature.dimension} leads to a differential inequality

\begin{equation}\label{eq.int.curvature.dimension}
\int_{\R^n}\Gamma_2(P_t f)d\gamma_n\geq \int_{\R^n}\Gamma(P_t f)d\gamma_n  \Leftrightarrow 2I+I'\leq0,\quad t\geq 0
\end{equation}

\noindent which can be easily integrated between $s$ and $t$ (with $0\leq s\leq t$). That is

\[
I(t)e^{2t}\leq I(s)e^{2s}.
\]

\noindent It is now classical to let $s\to0$ to easily recover Poincar\'e's inequality \eqref{eq.poincare} for the measure $\gamma_n$. As we will see in the next section, we will show that a differential inequality of the form

\begin{equation}\label{eq.inverse.curvature.dimension}
I'\geq -2(I+\psi),
\end{equation}

\noindent for some function $\psi$, can be used to obtain relevant bound (with respect to superconcentration phenomenon) on the variance of the function $f$ (being fixed) by letting $s$ fixed and $t\to+\infty$.

\begin{rem}
Let us make few remarks.

\begin{enumerate}
\item As it is proved in \cite{BGL}, the integrated curvature dimension inequality \eqref{eq.int.curvature.dimension} is, in fact, equivalent to the Poincar\'e's inequality \eqref{eq.poincare}.\\
\item As we will see in the next section, the inequality $I'\geq -2(I+\psi)$ is equivalent to a inverse, integrated, curvature dimension inequality which seems to be new. However, notice that the major difference between \eqref{eq.int.curvature.dimension} and \label{eq.inverse.curvature.dimension} is that the first one holds for a large class of function whereas the second is only true for a particular function $f$ (and $\psi$ depends on $f$).
\end{enumerate}
\end{rem}

\section{Inverse, integrated, curvature inequality}

In this section, we will use the methodology exposed in the preceding section to obtain variance bounds for a (fixed) function $f$ satisfying a inverse, integrated, curvature inequality $IC_{\gamma_n}(1,\psi)$.\\

 First, let us state a definition. We want to highlight the fact that this definition will be stated in a Gaussian framework $(\R^n,\Gamma, \gamma_n)$ with $\Gamma$ associated to the infinitesimal generator $L=\Delta-x\cdot \nabla$ and the Ornstein-Uhlenbeck's semigroup $(P_t)_{t\geq 0}$. The next definition can be extended, mutatis mutandis, to fit  the general framework of \cite{BGL}.

\begin{defn}
Let $f\,:\,\R^n\to \R$ be a smooth function. We said that $f$ satisfy a inverse, integrated, curvature criterion with function $\psi\,:\,\R_+\to \R$ if 

\begin{equation}\label{G1}
\int_{\R^n}\Gamma_2(P_t f)d\gamma_n\leq  \int_{\R^n}\Gamma(P_t f)d\gamma_n+\psi(t),\quad t\geq 0
\end{equation}

\noindent When the previous inequality is satisfied we denote it by $f\in IC_{\gamma_n}(1,\psi)$.
\end{defn}
\begin{rem}
\begin{enumerate}
\item Notice, again, that the inequality \eqref{G1} holds, a priori, only for the function $f$.\\
\item In the general framework of \cite{BGL}, we would say that $f\in IC_{\mu}(\rho, \psi)$, with $\mu$ a probability measure and $\rho\geq 0$, if and only if

\[
\int_{E}\Gamma_2(P_t f)d\mu\leq \rho\bigg( \int_{E}\Gamma(P_t f)d\mu+\psi(t)\bigg),\quad t\geq 0
\]

for the Markov triple $(E,\Gamma,\mu)$.\\
\item At some point, in the framework of the superconcentration's theory, it is implicitly assumed that an integrated curvature dimension also holds. That is to say, for every $f$ belonging to a nice set of function (which is stable under the action of the semigroup), the following holds

\[
\int_E\Gamma_2(f)d\mu\geq \rho \int_E\Gamma(f)d\mu, \quad \rho>0.
\]

\noindent Therefore $\mu$ satisfied a Poincar\'e's inequality with constant $\rho$.
\end{enumerate}
\end{rem}

Here we recall the statement of our main result.

\begin{thm}\label{prop.courbure.dimension.inverse}

Let  $f\,:\,\R^n\to \R$ be a smooth function. Assume that the following holds 

\begin{enumerate}
\item $f\in IC_{\gamma_n}(1,\psi)$ for some function $\psi\,:\,\R_+\to\R$.\\
\item
\[
\int_0^\infty e^{-2t}\int_t^\infty e^{2 s}\psi(s)ds dt<\infty.
\]
\end{enumerate}
Then,

$$
{\rm Var}_{\gamma_n}(f)\leq \bigg|\int_{\R^n}\nabla fd\gamma_n\bigg|^2+4\int_0^\infty e^{-2 t}\int_t^\infty e^{2s}\psi(s)dsdt
$$ 
\noindent with $|\cdot|$ the standard Euclidean norm.
\end{thm}

\begin{proof}
The inverse curvature condition $IC_{\gamma_n}(1,\psi)$ (equation \ref{G1}) is equivalent to the following differential inequality :

\begin{equation}\label{eq.diff1}
I'\geq -2(I+\psi),
\end{equation}

\noindent where $I(t)=\int_{\R^n}|\nabla P_t f|^2d\gamma_n$, $t\geq 0$. Set $I(t)=K(t)e^{-2t}$, inequality \eqref{eq.diff1} becomes

\begin{equation}\label{eq.diff2}
K'(t)\geq -2e^{2t}\psi(t),\quad t\geq 0
\end{equation}

\noindent Now, integrate inequality \eqref{eq.diff2} between $s$ and $t$. That is  

$$
K(t)-K(s)\geq -2\int_s^t e^{2u}\psi(u)du,\quad \text{for all} \quad0\leq s\leq t.
$$

\noindent Then, let $t\to\infty$, this yields
 
$$
K(s)\leq \big[\lim_{t\to\infty}K(t)\big]+2\int_s^\infty e^{2u}\psi(u)du,\quad s\geq 0,
$$

\noindent To conclude, observe that

$$
K(t)=I(t)e^{2t}\to_{t\to \infty}\bigg|\int_{\R^n}\nabla fd\gamma_n\bigg|^2
$$

\noindent by ergodicity of $(P_t)_{t\geq 0}$.  Finally, we have, for every $t\geq 0$,

\begin{equation}\label{G2}
I(t)=\int_{\R_n}\Gamma(P_t f)d\gamma_n\leq e^{-2t} \bigg(\bigg|\int_{\R^n}\nabla fd\gamma_n\bigg|^2+2\int_t^\infty e^{2s}\psi(s) ds\bigg).
\end{equation}

\noindent It suffices to use the dynamical representation of the variance \eqref{eq.representation.variance.gaussienne} with elementary calculus to end the proof.

\end{proof}

\begin{rem}
This method of interpolation, between $t$ and $+\infty$, has also been used in \cite{KT3} in order to obtain Talagrand's inequality of higher order.
\end{rem}

\subsection{Another Variance bound}

As we will see in the last section, it is sometimes useful to restrict an $IC_{\mu}(1,\psi)$, for some probability measure $\mu$, up to a time $T$ in order to improve the dependance with respect to some parameter.\\

 In other words, the setting is the following : assume that an $IC_\mu(1,\psi)$ holds and that we are able to produce some $T>0$ such that the bound of $I(T)$ (given by the equation \eqref{G2}) is particularly nice (with respect to some parameter). Now, we have to bound the variance in a different manner in order to use the information on $I(T)$. To this task, we will prove the next proposition.






\begin{prop}\label{prop.partial.inverse.curvature}
Let $f\,:\,\R^n\to\R$ be a function smooth enough. Then, for any $T>0$


\[
{\rm Var}_{\gamma_n}(f)\leq \frac{2TI(0)}{1-e^{-2T}}\bigg[\frac{1}{\log a}-\frac{1}{a\log a}\bigg]
\]

\noindent with $a=\frac{I(0)}{I(T)}$ and $I(t)=\int_{\R^n}\Gamma(P_tf)d\gamma_n$.
\end{prop}
\begin{rem}
This proposition will be used to show that the Free Energy is superconcentrated for some Spin Glasses models. Although we stated the preceding Proposition \ref{prop.partial.inverse.curvature} for the standard Gaussian measure $\gamma_n$, it will also hold (up to obvious renormalization) for $\mu$ the law of a centered Gaussian vector with covariance matrix $M$.
\end{rem}

To prove the preceding theorem, we will need two further arguments. \\

First, we present an inequality due to Cordero-Erausquin and Ledoux \cite{CoLed}. The proof of this inequality rests on the fact that the Poincar\'e's inequality satisfied by $\gamma_n$ implies an exponential decay of the variance along the semigroup $(P_t)_{t\geq 0}$.

\begin{lem}\label{lem.cordero.ledoux}[Cordero-Erausquin-Ledoux]

Let $f\,:\,\R^n\to\R$ be a function smooth enough. Then, for any $T>0$, the following holds

\begin{equation}
{\rm Var}_{\gamma_n}(f)\leq \frac{2}{1-e^{-2T}}\int_0^TI(t) dt
\end{equation}

\noindent with $I(t)=\int_{\R^n}\Gamma(P_tf)d\gamma_n$.
\end{lem}

\begin{proof}
For the sake of completeness we give the proof of the preceding Lemma.
\begin{eqnarray*}
{\rm Var}_{\gamma_n}(f)&=&\E_{\gamma_n}[f^2]-\E_{\gamma_n}[(P_Tf)^2]+\E_{\gamma_n}[(P_Tf)^2]-\E_{\gamma_n}[P_Tf]^2\\
&=&-\int_0^T\frac{d}{ds}\E_{\gamma_n}[(P_sf)^2]ds+{\rm Var}_{\gamma_n}(P_Tf)\\
&\leq & 2\int_0^TI(s)ds+e^{-2T}{\rm Var}_{\gamma_n}(f).
\end{eqnarray*}
\end{proof}

Secondly, we will use the fact that the  infinitesimal  generator $(-L)$ of the Ornstein-Uhlenbeck process $(X_t)_{t\geq 0}$ admits a (discrete) spectral decomposition. Then, denote by $dE$ the spectral measure. According to \cite{BGL}, this leads to a different representation of $t\mapsto I(t)$. With $f\,:\,\R^n\to\R$ being fixed, we have :

$$
I(t)=\int_0^\infty e^{-2t} dE(f),\quad t\geq 0
$$

\noindent As it is proven in \cite{BauWan} (with the preceding representation),  $t\mapsto I(t)$ satisfies an H\"older-type inequality. That is to say, for every $T>0$,
\begin{lem}\label{lem.baudoin.wang}[Baudoin-Wang]
\begin{equation}\label{eq.holder.spectrale}
I(s)\leq I(0)^{1-s/T}I(T)^{s/T},\quad 0\leq s\leq T
\end{equation}
\end{lem}

\noindent Now, we can prove Proposition \ref{prop.partial.inverse.curvature} with the help of preceding Lemma.

\begin{proof}(of Proposition \ref{prop.partial.inverse.curvature})
First use Lemma \ref{lem.cordero.ledoux} to get 

\[
{\rm Var}_{\gamma_n}(f)\leq \frac{2}{1-e^{-2T}}\int_0^TI(t) dt.
\]

Then, use Lemma \ref{lem.baudoin.wang}. This yields

\begin{eqnarray*}
{\rm Var}_{\gamma_n}(f)&\leq& \frac{2}{1-e^{-2T}}\int_0^TI(0)^{1-s/T}I(T)^{s/T} dt\\
&=&\frac{2I(0)}{1-e^{-2T}}\int_0^Te^{-\frac{s}{T}\log a}dt
\end{eqnarray*}

\noindent where $a=\frac{ I(0)}{I(T)}\geq 1$ and $I(t)=\int_{\R^n}\Gamma(P_tf)d\gamma_n$. Finally, 
elementary calculus ends the proof.

\end{proof}

\section{Application in Spin Glasses's theory}

\subsection{Introduction}

We begin by a short introduction to the Theory of Spin Glass (cf. \cite{Talspin1, Talspin2, Bov1} for more details). \\

Most of the time, in Spin Glasses Theory, it is customary to consider a centered Gaussian field $\big(H_n(\sigma)\big)_{\sigma\in \{-1,1\}^n}$ on the discrete cube $\{-1,1\}^n$ (the map $\sigma\mapsto H_n(\sigma)$ is called the Hamiltonian of the system) and to focus on $\max_{\sigma\in\{-1,1\}^n}H_n(\sigma)$ (or  $\min_{\sigma\in\{-1,1\}^n}H_n(\sigma)$). In general, this quantity is rather complex and presents a lack of regularity. Therefore, one focus on a smooth approximation of the maximum (or the minimum) called the Free Energy $F_{n,\beta}$. This function is defined as follow

\[
F_{n,\beta}=\pm\frac{1}{\beta}\log \bigg(\sum_{\sigma\in\{-1,1\}^n}e^{\pm \beta H_n(\sigma)}\bigg)
\]

\noindent where $\beta>0$ corresponds to the (inverse) of the temperature and its sign depends on whether you want to study the maximum or the minimum of $H_n$ over the discrete cube.\\

For instance, for the REM, we have 
\[
H_n(\sigma)=\sqrt{n}X_\sigma,\quad \sigma\in\{-1,1\}^n
\]
\noindent  with $(X_\sigma)_{\sigma\in\{-1,1\}^n}$ is a sequence of i.i.d. standard Gaussian random variables.\\

For the SK Model, the Hamiltonian is more complex, 
\[
H_n(\sigma)=-\frac{1}{\sqrt{n}}\sum_{i,j=1}^nX_{ij}\sigma_i\sigma_j,\quad \sigma\in\{-1,1\}^n
\]
\noindent  with $(X_{ij})_{1\leq i,j\leq n}$ is a sequence of i.i.d. standard Gaussian random variables.\\

 In the remaining of this section, we will show how inverse, integrated, curvature inequality \eqref{G1} can provide relevant bounds on the variance of $F_{n,\beta}$. We will focus on the REM and the SK Model.  For the remaining of this note we will denote by $f_{\beta}$, for $\beta>0,$ the following function 
 
 \[
 f_{\beta}(x)=\frac{1}{\beta}\log \big(\sum_{i=1}^ne^{\beta x_i}\big), \quad x=(x_1,\ldots,x_n)\in\R^n
 \]

\subsection{Random Energy Model}

In this section we will show how Theorem \ref{prop.courbure.dimension.inverse} is useful to obtain relevant bound on the variance of the Free Energy $F_{n,\beta}$ (with $\beta$ close to $0$) for the REM.

\begin{prop}\label{prop.rem.inverse.curvature}
For any $\beta>0$, $f_\beta\in IC_{\gamma_n}(1,\psi)$ with 

\[
\psi(t)=2\beta^2e^{-2t} I(t)
\]

\noindent where, let us recall it, $I(t)=\int_{\R^n}\Gamma(P_t f_\beta)d\gamma_n$ and $\Gamma$ is the standard "carr\'e du champ" operator.
\end{prop}

\noindent We will need the following Lemma to prove the preceding Proposition.

\begin{lem}\label{lem.courbure.dimension}
Let $(u_j)_{j=1,\ldots,n}$ be a family of function, with $u_i\,:\,\R^n\to\R$ for any $j=1,\ldots,n$, satisfying the following condition.

\begin{itemize}
\item $0\leq u_j(x)\leq 1$ $\forall j=1,\ldots,n,\forall x\in\R^n $.\\
\item $\sum_{j=1}^nu_j(x)\leq 1$ $\forall x\in\R^n$.\\
\end{itemize}

Then, for any function  $v\,:\,\R^n\to\R_+$ and any probability measure $\mu$, we have

$$
\sum_{j=1}^n\bigg(\int_{\R^n}u_j(x) v(x)d\mu(x)\bigg)^2\leq \bigg(\int_{\R^n}vd\mu\bigg)^2
$$

\end{lem}

\begin{proof}

\noindent Fubini's Theorem implies that
 
$$
\sum_{j=1}^n\bigg(\int_{\R^n}u_j(x) v(x)d\mu(x)\bigg)^2=\sum_{j=1}^n\int_{\R^n}\int_{\R^n}u_j(x)u_j(y)v(x)v(y)d\mu(x)d\mu(y).
$$

\noindent Therefore,

\begin{eqnarray*}
\sum_{j=1}^n\bigg(\int_{\R^n}u_j(x) v(x)d\mu(x)\bigg)^2&\leq&\sum_{j=1}^n\int_{\R^n}\int_{\R^n}u_j(x)v(x)v(y)d\mu(x)d\mu(y)\\
 &\leq&\int_{\R^n}\int_{\R^n}v(x)v(y)d\mu(x)d\mu(y)\\
&=&\bigg(\int_{\R^n}v(x)d\mu(x)\bigg)^2
\end{eqnarray*}

\end{proof}

\noindent Now we turn to the proof of Proposition \ref{prop.rem.inverse.curvature}.

\begin{proof}(Proposition \ref{prop.rem.inverse.curvature}).

First, observe that the condition $IC_{\gamma_n}(1,\psi)$ is equivalent to

$$
\int_{\R^n}\Gamma_2\big(P_t(f_\beta)\big)d\gamma_n\leq (1+2\beta^2e^{-2t})\int_{\R^n}\Gamma\Big(P_t (f_\beta)\big)d\gamma_n,\quad t\geq 0.
$$

\noindent That is (since $\Gamma_2(f)=\|{\rm Hess}f\|_2^2$ and $\Gamma(f)=|\nabla f|^2$)

\begin{equation}\label{eq.courbure.dimension.integree.inverse.energie.libre}
\int_{\R_n}\|{\rm Hess} P_t(f_\beta)\|_2^2d\gamma_n\leq 2\beta^2e^{-2t}\int_{\R^n}|\nabla P_t (f_\beta)|^2d\gamma_n,\quad t\geq0.
\end{equation}

\noindent Now, observe that, pointwise, equation \eqref{eq.courbure.dimension.integree.inverse.energie.libre} is equivalent to (thanks to the commutation property between $\nabla$ and $(P_t)_{t\geq 0}$)

$$
\sum_{i,j=1}^n[P_t(\partial^2_{ij}f_\beta)]^2\leq 2\beta^2\sum_{i=1}^n [P_t(\partial_i f_\beta)]^2,\quad \forall t\geq 0
$$

\noindent Elementary calculus yields, for every $i=1,\ldots,n$, and every $\beta>0$,

$$
\partial_if_\beta=\frac{e^{\beta x_i}}{\sum_{k=1}^n e^{\beta x_k}}
$$ 

\noindent and, for every $j=1,\ldots,n$, 

$$
\partial_j\partial_if_\beta=\beta(\partial_if_\beta\delta_{ij}-\partial_if_\beta\partial_j f_\beta).
$$

\noindent Thus, for every $t\geq 0$,

$$
\sum_{i,j=1}^n[P_t(\partial^2_{ij}f_\beta)]^2=\beta^2\sum_{i=1}^n\big[P_t(\partial_i f_\beta)\big]^2-2\beta\sum_{i=1}^nP_t(\partial_i f_\beta) P_t\big[(\partial_i f_\beta)^2\big]+\beta^2\sum_{i,j=1}^n\big[P_t(\partial_i f_\beta\partial_j f_\beta)\big]^2.
$$

\noindent First ignore the crossed terms (which are always non positive), then apply Lemma \ref{lem.courbure.dimension} to the third term. \\
\newline

Indeed, let $i\in\{1,\ldots,n\}$ be fixed and set $u_j=\partial_jf_\beta$ and $v=\partial_if_\beta$. Thus, Lemma \ref{lem.courbure.dimension} implies

$$
\sum_{j=1}^n \big[P_t(\partial_if_\beta\partial_jf_\beta)\big]^2\leq P_t^2(\partial_if\beta).
$$

\noindent This inequality finally yields, 

$$
\sum_{i,j=1}^n[P_t(\partial^2_{ij}f_\beta)]^2\leq \beta^2\sum_{i=1}^n\big[P_t(\partial_i f_\beta)\big]^2+\beta^2\sum_{i,j=1}^n\big[P_t(\partial_i f_\beta\partial_j f_\beta)\big]^2 \leq 2\beta^2\sum_{i=1}^n\big[P_t(\partial_i f_\beta)\big]^2.
$$

\end{proof}

Now, the criterion $IC_{\gamma_n}(1,\psi)$ gives the following bound on the variance of $F_{n,\beta}$.

\begin{prop}\label{prop.rem}
For $\beta\in\big(0,\sqrt{\frac{\ln 2}{2}}\big)$, we have 

\[
{\rm Var}_{\gamma_n}(F_{n,\beta})\leq \bigg(\frac{1-\beta^2}{1-2\beta^2}\bigg)\frac{1}{n}.
\]

For $\beta\sim\sqrt{\log n}$, we have 
\[
{\rm Var}_{\gamma_n}(F_{n,\beta})\leq \frac{C}{\log n}.
\]
\noindent with $C>0$ a universal constant.
\end{prop}
\begin{rem}
These bounds has to be compared with the results exposed in \cite{BovKurLow, Chatt1} (be careful with the different renormalization). In \cite{BovKurLow}, it is shown that 

\[
{\rm Var}_{\gamma_n}(F_{n,\beta})\sim \frac{C(\beta)}{n},\quad \beta<\sqrt{\frac{\log n}{2}}
\]

\noindent with $C(\beta)=\frac{e^{\beta^2}}{\beta^2}(1-e^{-\beta^2})$. Despite the wrong dependance in $\beta$, we recover the right order of magnitude in $n$ in this temperature regime.\\

On the contrary, in the low temperature regime (when $\beta\sim \sqrt{\log n}$, cf. \cite{Chatt1, BovKurLow}), the variance of $F_{n,\beta}$ is of the same order as the variance of the maximum of i.i.d. standard Gaussian random variable. That is of order $C/\log n$.
\end{rem}

\begin{proof}
As it will be useful in the sequel, observe that (by symmetry) the following holds

\[
\int_{\R^n}\partial_if_\beta d\gamma_n=\frac{1}{n},\quad \forall i=1,\ldots,n.
\]

Now, let $\beta>0$ and use Theorem \ref{prop.courbure.dimension.inverse} which implies that

\begin{equation}\label{eq.free.energy}
{\rm Var}_{\gamma_n}(F_{n,\beta})\leq \frac{1}{n}+4\beta^2\int_0^\infty e^{-2s}(1-e^{-2s})\sum_{i=1}^n\int_{\R^n}P_s^2(\partial_if_\beta)d\gamma_nds
\end{equation}

\noindent where we used Fubini's Theorem and the commutation property between $\nabla$ and $P_s$.\\

For the first bound, when $\beta\in\big(0,\frac{\sqrt{2}}{2})$, it is possible to rewrite (thanks to the dynamical representation of the variance \eqref{eq.free.energy}) the integral in the right hand side as

\[
2\beta^2 {\rm Var}_{\gamma_n}(F_{n,\beta})-4\beta^2\int_0^\infty\sum_{i=1}^n\int_{\R^n}P_s^2(\partial_if_\beta)d\gamma_nds
\]

Furthermore, by Jensen's inequality and the invariance of $(P_t)_{t\geq 0}$ with respect to $\gamma_n$, we have 

\[
\int_{\R^n}P_s^2(\partial_if_\beta)d\gamma_n\geq \bigg(\int_{\R^n}P_s(\partial_if_\beta)d\gamma_n\bigg)^2=\frac{1}{n^2},\quad \forall i=1,\ldots,n,\quad \forall s>0
\]

Thus, ${\rm Var}_{\gamma_n}(F_{n,\beta})\leq\bigg(\frac{1-\beta^2}{1-2\beta^2}\bigg)\frac{1}{n}$.\\

For the second bound, we will use the inequality \eqref{eq.free.energy} together with hypercontractive estimates of $(P_t)_{t\geq 0}$ (cf. \cite{Chatt1,KT1, KT2, CoLed}). More precisely, we have 

\[
\|P_s(\partial_i f_\beta)\|_2^2\leq \|\partial_if_\beta\|^2_{1+e^{-2s}},\quad \forall i=1,\ldots,n,\quad \forall s>0
\]

It is then standard, cf. section 4 in \cite{KT2} for instance, to prove that 

\[
\int_0^\infty e^{-2s}(1-e^{-2s})\|\partial_if_\beta\|_{1+e^{-2s}}^2ds\leq \frac{C\|\partial_i f_\beta\|_2^2}{\big[1+\log \frac{\|\partial_i f\beta\|_2}{\|\partial_i f_\beta\|_1}\big]^2}
\]

\noindent where $C>0$ is a numerical constant. Then, it is elementary to conclude.
\end{proof}

\subsection{SK Model}
In this section we show how some work of Chatterjee (from \cite{Chatt1}) can be rewritten  in term of an inverse, integrated, curvature criterion. Then, it allows us to easily recover a bound, obtained by Talagrand (cf. \cite{Talspin1,Talspin2}), on the variance of the Free Energy for the SK model at high temperature.\\

First, we need to express the $\Gamma$ and $\Gamma_2$ operator when $\gamma_n$ is replaced by $\mu$ the law of a centered Gaussian vector, in $\R^n$, with covariance matrix $M$. \\

Let $X$ be a random Gaussian vector with $\mathcal{L}(X)=\mu$ and consider $Y$ an independant copy of $X$. It is then possible to define the generalized Ornstein-Uhlenbeck process, which we will still denote by $(X_t)_{t\geq 0}$, as follow

\[
X_t=e^{-t}X+\sqrt{1-e^{-2t}}Y,\quad t\geq 0
\]

Similarly, we also denote by $(P_t)_{t\geq 0}$ the associated semigroup. Then, it is known (cf. \cite{Chatt1, KT, KT2}) that, for any smooth function $f\,:\, \R^n \to\R$,

\[
I(t)=\int_{\R^n}\Gamma(P_t f)d\mu=2\int_{\R^n}e^{-2t}\sum_{i,j}M_{ij}(\partial_if) P_t(\partial_j f)d\mu,\quad t\geq 0
\] 

As we will see latter, it will be more convenient to work with 

\[
I_r(t)=2\int_{\R^n}e^{-2t}\sum_{i,j}(M_{ij})^r(\partial_if) P_t(\partial_j f)d\mu,\quad t\geq 0
\]
\noindent where $r$ is a positive integer. In the rest of this section, we choose $f=f_\beta$.

\begin{prop}[Chatterjee]
Assume that $M_{ij}\geq 0$ for all $(i,j)\in\{1,\ldots,n\}^2$. Then, for any $t\geq 0$, the following holds
\begin{equation}\label{eq.chatterjee}
I'_r(t)\geq -2\big[I_r(t)+2\beta^2e^{-2t}J_{r+1}(t)\big]
\end{equation}
\noindent with $J_r(t)=e^{2t}I_r(t)$.
\end{prop}
\begin{rem}
\begin{enumerate}
\item In \cite{Chatt1}, Chatterjee proved that $J'_r(t)\geq -4\beta^2e^{-2t}J_{r+1}(t)$ for any $r\in\N^{\ast}$. The proof is similar the proof of Lemma \ref{lem.courbure.dimension} with the additional use of H\"older's inequality.\\
\item In particular, when $r=1$, Chatterjee's proposition amounts of saying that 

\[
f_\beta\in IC_{\mu}(1,\psi)
\]

\noindent with $\psi(t)=2\beta^2e^{-2t}J_2(t)$. Unfortunately, it remains hard to upper bound this quantity by something relevant.
\end{enumerate}
\end{rem}

As observed in the preceding remark, the inverse, integrated, curvature criterion can not be used in the present form. However, it is possible to recycle the arguments of section three. That is, use $l$ times, with $l\in\N$, the fundamental Theorem of analysis (on $t\mapsto I_r(t)$) together with the inequality \eqref{eq.chatterjee} and let $l\to+\infty$. This leads to a useful bound on the function $t\mapsto I_r(t)$ for any $r\in \N^\ast$.

\begin{prop}[Chatterjee]
Assume that $M_{ij}\geq 0$ for all $(i,j)\in\{1,\ldots,n\}^2$. Then, for any $t\geq 0$, the following holds

\begin{equation}\label{eq.chatterjee.2}
I_r(t)\leq e^{-2t} \sum_{i,j=1}^n(M_{ij})^re^{2\beta^2 e^{-2t}M_{ij}}\nu_i\nu_j,\quad \forall r\geq 1
\end{equation}

\noindent where $\nu_i=\int_{\R^n}\partial_if_\beta d\mu$ for all $i=1,\ldots,n$.

\end{prop}
\begin{rem}
When $r=1$, the main step of Chatterjee's proof is equivalent to show that $f_\beta\in IC_\mu(1, \psi)$ with $\psi(t)=2\beta^2e^{-2t}\sum_{i,j=1}M_{ij}e^{2\beta^2e^{-2t}M_{ij}}\nu_i\nu_j$.
\end{rem}

Unfortunately, the repeated use of the differential inequality \eqref{eq.chatterjee} degrades the upper bound on $t\mapsto I_r(t)$. As we will briefly see in the next subsection, Chatterjee used equation \eqref{eq.chatterjee.2} only for a fixed $T>0$ (large enough).  We show, in the next Proposition, that this bound (for $r=1$) is still relevant to recover some work of Talagrand on the variance of $F_{\beta,n}$, with small $\beta$, for the SK model (cf. \cite{Talspin1,Talspin2}).

\begin{prop}
Let $M$ be the covariance structure of the SK model. Then, for any $\beta\in\big(0,\frac{1}{2}\big)$, the following holds

\begin{equation}\label{eq.free.energy.sk}
{\rm Var}_\mu(F_{n,\beta})\leq \frac{C_\beta}{2\beta^2}
\end{equation}

\noindent with $C_\beta>0$ is a universal constant which does not depend on $n$ (only on $\beta$).
\end{prop}
\begin{proof}
First we show that inequality \eqref{eq.chatterjee.2} leads to a general upper bound on the variance of $F_{n,\beta}$ which might be of independant interest. Then, we choose $M$ to be the covariance structure of the $SK$ model and proved inequality \eqref{eq.free.energy.sk}.\\

When $r=1$, Proposition \eqref{eq.chatterjee.2} combined with equation \eqref{eq.variance.represantation2} implies that, for any $\beta>0$, 

\begin{eqnarray*}
{\rm Var}_\mu(F_{n,\beta})&\leq& 2\int_0^\infty e^{-2t}\sum_{i,j=1}^n M_{ij}e^{2\beta^2e^{-2t}M_{ij}}\nu_i\nu_jdt\\
&\leq& \frac{1}{2\beta^2}\sum_{i,j=1}^ne^{2\beta M_{ij}}\nu_i\nu_j
\end{eqnarray*}

Following Chatterjee (cf. \cite{Chatt1}), choose $M$ to be the covariance structure of the SK model. That is, 

\[
M_{\sigma \sigma'}=\bigg(\frac{1}{\sqrt n}\sum_{i=1}^n\sigma_i\sigma_i'\bigg)^2,\quad \forall \sigma,\sigma'\in\{-1,1\}^n.
\]

Besides, observe (by symmetry) that, for each $\sigma\in\{-1,1\}^n,$ 
\[
\nu_{\sigma}=\E_\mu\big[\partial_\sigma F_{n,\beta}\big]=\frac{1}{2^n}.
\]

Thus, 
\[
{\rm Var}_\mu(F_{n,\beta})\leq \frac{1}{2\beta^2}\E_{\sigma,\sigma'}\big[e^{2\beta^2\big(\frac{1}{\sqrt{n}}\sigma_i\sigma'_i\big)^2}\bigg]
\]

Finally, if $\beta\in\big(0,\frac{1}{2}\big)$ we have $\E_{\sigma,\sigma'}\big[e^{2\beta^2\big(\frac{1}{\sqrt{n}}\sigma_i\sigma'_i\big)^2}\bigg]=C(\beta)$.
\end{proof}
\begin{rem}
\begin{enumerate}
\item Talagrand obtained such upper bound on the variance, for $0<\beta<1$, as a consequence of precise (and much harder to prove than our variance bounds) concentration inequalities for the Free Energy. \\
\item Preceding result can also be used to show that the ground states of the SK model is superconcentrated. Indeed, since $\|f_\beta-\max_{i=1,\ldots,n}\|_{\infty}\leq \frac{\log n}{\beta}$ for all $\beta>0$, we have

$$
{\rm Var}_\mu\big(\max_{\sigma\in \{-1,1\}^n}H_n(\sigma)\big)\leq 3{\rm Var}_\mu(F_{n,\beta})+6\bigg(\frac{\log n}{\beta}\bigg)^2,\,\beta>0
$$

\noindent Then, choose $\beta=1/4$, this yields ${\rm Var}_\mu\big(\max_{\sigma\in\{-1,1\}^n}H_n(\sigma)\big)\leq C_{\beta}(\log n)^2$ which improve upon the bound given by Poincar\'e's inequality. \\
\end{enumerate}
\end{rem}
\subsection{Improvements of Variance bounds with respect to the parameter $\beta$}

Let us collect some results of Chatterjee and briefly explain how Proposition \ref{prop.partial.inverse.curvature} can be used to improve the dependance of the variance bounds with respect to $\beta$. However, the dependance in $n$ will be worse.\\

 Chatterjee used, in \cite{Chatt1}, a Theorem of Bernstein about completely monotone function. As far as we are concerned, the spectral framework exposed in section three  seems to be more natural to work with and provides equivalent results. \\

The arguments, in order to improve the dependance in $\beta$, can be summarize as follow : choose $T$ such that $I(T)$ can be bounded by a relevant quantity and apply Proposition \ref{prop.partial.inverse.curvature}.

\begin{prop}[Chatterjee]
In the $SK$ model the following holds

\[
{\rm Var}_\mu(F_{n,\beta})\leq \frac{C_1n\log (2+C_2\beta)}{\log n},\quad \forall \beta>0
\]
\noindent with $C_1,C_2>0$ two numerical constants.

\end{prop}
\begin{rem}
Here $T>0$ is choosen  such that 
\[
\E_{\sigma,\sigma'}\bigg[M_{\sigma\sigma'}e^{2\beta^2e^{-2T}M_{\sigma\sigma'}}\bigg]=C_\beta,\quad \forall \beta>0
\]

\noindent where $M_{\sigma\sigma'}=\big(\frac{1}{\sqrt{n}}\sum_{i=1}^n\sigma_i\sigma_i'\big)^2$ and $C_\beta>0$ is a constant that does not depend on $n$. That is $T=\frac{1}{2}\log\big(\frac{2\beta^2}{\gamma}\big)$ for some sufficiently small constant $\gamma>0$ (cf. \cite{Chatt1}).
\end{rem}

\begin{prop}[Chatterjee]
In the REM, the following holds for $\beta>2\sqrt{\log 2}$,

\[
{\rm Var}_\mu(F_{n,\beta})\leq C_\beta
\]

\noindent where $C_\beta>0$ is a constant that does no depend on $n$.
\end{prop}
\begin{rem}
Here $T$ is choosen as $T=\frac{1}{2}\log (2\beta ^2)$ so that $I(T)\leq \frac{n}{2^n}e^{-2T}e^n$ and the upper bound is relevant in the low temperature regime (cf. \cite{Chatt1, BovKurLow}). Notice the difference of renormalization  with Proposition \ref{prop.rem} (one has to replace the number of random variables $n$ by $2^n$ and the i.i.d. standard Gaussian random variables $(X_i)_{i=1,\ldots,2^n}$ by $\sqrt{n}X_i$ in the Proposition). 
\end{rem}
\textit{Aknowledgment\,:\, I thank M. Ledoux for fruitful discussions on this topic.}

\bigskip

\end{document}